\documentclass[11pt,reqno]{amsart}
\usepackage[margin=2.35cm]{geometry}
\usepackage{amsmath,amssymb,mathtools}
\usepackage{enumitem}
\usepackage{microtype}
\usepackage[pagebackref,hidelinks]{hyperref}
\usepackage[nameinlink,capitalise,noabbrev]{cleveref}

\renewcommand*{\backrefalt}[4]{%
  \ifcase #1\relax
  \else \space\(\hookleftarrow\)\,#2%
  \fi}

\numberwithin{equation}{section}

\makeatletter
\newenvironment{lefttagged}{\tagsleft@true}{\tagsleft@false}
\renewcommand{\@setauthors}{%
  \begingroup
  \trivlist
  \centering\normalsize \@topsep30\p@\relax
  \advance\@topsep by -\baselineskip
  \item\relax
  \authors
  \endtrivlist
  \endgroup
}
\makeatother

\newtheorem{theorem}{Theorem}[section]
\newtheorem{lemma}[theorem]{Lemma}
\newtheorem{proposition}[theorem]{Proposition}
\newtheorem{corollary}[theorem]{Corollary}
\newtheorem{conjecture}[theorem]{Conjecture}

\newtheorem*{claim*}{Claim}
\newtheoremstyle{nestedclaimstyle}
  {\topsep}{\topsep}{\itshape}{}{\bfseries}{:}{0.5em}{}
\theoremstyle{nestedclaimstyle}
\newtheorem{nestedclaim}{Claim}[theorem]
\theoremstyle{plain}

\theoremstyle{definition}
\newtheorem{definition}[theorem]{Definition}
\newtheorem{example}[theorem]{Example}
\theoremstyle{remark}

\crefname{theorem}{Theorem}{Theorems}
\crefname{lemma}{Lemma}{Lemmas}
\crefname{proposition}{Proposition}{Propositions}
\crefname{corollary}{Corollary}{Corollaries}
\crefname{claim}{Claim}{Claims}
\crefname{nestedclaim}{Claim}{Claims}
\Crefname{nestedclaim}{Claim}{Claims}
\crefname{fact}{Fact}{Facts}
\crefname{definition}{Definition}{Definitions}
\crefname{example}{Example}{Examples}
\crefname{conjecture}{Conjecture}{Conjectures}
\crefname{remark}{Remark}{Remarks}

\newenvironment{claimblock}
  {\begin{list}{}{%
     \setlength{\leftmargin}{2em}%
     \setlength{\rightmargin}{2em}%
     \setlength{\itemsep}{0pt}}\item[]}
  {\end{list}}

\newcommand{\cC}{\mathcal C}

\newcommand{\cP}{\mathcal P}

\newcommand{\abs}[1]{\lvert#1\rvert}
\newcommand{\set}[1]{\{#1\}}
\newcommand{\floor}[1]{\lfloor#1\rfloor}
\newcommand{\ceil}[1]{\lceil#1\rceil}

\title[Chv\'atal-type codegree condition]
{A Chv\'atal-type codegree condition for Hamiltonian cycles in $k$-uniform hypergraphs}
\author[Lu-Ming Zhang]{Lu-Ming Zhang\textsuperscript{*}}
\thanks{\textsuperscript{*}Department of Mathematics, London School of Economics and Political Science, London, United Kingdom (\href{mailto:l.zhang100@lse.ac.uk}{l.zhang100@lse.ac.uk}).}
\date{}

\begin{document}

\begin{abstract}
We prove an asymptotic Chv\'atal-type codegree criterion for tight
Hamiltonian cycles in $k$-uniform hypergraphs for every fixed $k\ge3$.
The criterion allows small degrees to be compensated by large degrees in the link graphs of $(k-2)$-tuples. 
It strengthens the asymptotic
Dirac-type theorem by R\"odl, Ruci\'nski and Szemer\'edi and implies a P\'osa-type criterion conjectured by Sch\"ulke.

%We establish a Chv\'atal-type codegree criterion for tight Hamiltonian cycles in uniform hypergraphs.  
%This gives a hypergraph analogue of Chv\'atal's classical degree-sequence theorem.
%In particular, it strengthens the celebrated Dirac-type result by R\"odl, Ruci\'nski, and Szemer\'edi and implies, as a corollary, a P\'osa-type result conjectured by Sch\"ulke.
%The proof uses the Hamilton-framework method of Lang and Sanhueza-Matamala, verifying the framework axioms and establishing inheritance of the codegree conditions by many small induced subgraphs.
\end{abstract}

\maketitle

\section{Introduction}

We say a (simple) graph is Hamiltonian if it contains a Hamiltonian cycle.
Determining whether a graph is Hamiltonian is NP-complete \cite{Karp}. 
Therefore, finding sufficient conditions for Hamiltonicity has long been a central problem in graph theory.  
Dirac \cite{Dirac} established a sharp minimum degree condition that every $n$-vertex graph $G$ with $\delta(G) \ge n/2$ is Hamiltonian.  
P\'osa \cite{Posa} strengthened this minimum-degree condition to a degree-sequence condition.
Let $G$ be a graph on vertex set $[n]$ with degree sequence $d(1)\le\cdots\le d(n)$.
P\'osa proved that if $d(i)\ge i+1$ for every $i<n/2$, then $G$ is Hamiltonian.  
Chv\'atal \cite{Chvatal} further strengthened P\'osa's theorem by allowing a failure of the lower-degree condition to be
offset by a corresponding upper-degree condition, as follows.

\begin{theorem}[Chv\'atal \cite{Chvatal}]\label{thm:classic}
For $n\ge 3$,
let $G$ be a graph on vertex set $[n]$ with degree sequence $d(1)\le\cdots\le d(n)$.  
Suppose that, for every integer $1\le i < n/2$,
\[
 d(i)\ge i + 1\quad\text{or}\quad d(n-i)\ge n-i.
\]
Then, $G$ contains a Hamiltonian cycle.
\end{theorem}

The problem of characterising Hamiltonicity under degree conditions generalises to hypergraphs.
In a $k$-uniform hypergraph ($k$-graph, for short), a \emph{(tight) Hamiltonian cycle} is a cyclic ordering of all vertices such that every $k$ cyclically consecutive vertices form an edge. 
Similarly, we say a $k$-graph is Hamiltonian if it contains a Hamiltonian cycle.
For a $k$-graph $G$ and $S\subseteq V(G)$, the degree of $S$ is $d_G(S):=\bigl|\{e\in E(G):S\subseteq e\}\bigr|$.
We omit the subscript $G$ when it is clear.
For $1\le d\le k-1$, the minimum $d$-degree of $G$ is
$\delta_{d}(G)=\min \set{d(S): S\in \binom{V(G)}{d}}$.
We also call $d(S)$ the codegree of $S$ for $|S| = k-1$ and $\delta_{k-1}(G)$ the minimum codegree of $G$.
In this work, we will focus on codegree conditions for tight Hamiltonicity.

The study of Dirac-type codegree conditions for tight Hamiltonian cycles in
hypergraphs was initiated by Katona and Kierstead
\cite{KatonaKierstead}.  R\"odl, Ruci\'nski, and Szemer\'edi \cite{RRSApprox} established,
for every fixed $k\ge3$, the asymptotically sharp minimum-codegree threshold $(1/2+o(1))n$ for $n$-vertex $k$-graphs.  
For $k=3$, they later determined the exact threshold
$\ceil{(n-1)/2}$ for all sufficiently large $n$ \cite{RRS3}.  
In 2026, Letzter,
Lang, Ranganathan, and Sanhueza-Matamala announced that, for every fixed $k\ge3$ and all sufficiently large $n$, the exact minimum-codegree
threshold is
$\ceil{(n-k+2)/2}$, confirming the conjecture of
Katona and Kierstead proposed in \cite{KatonaKierstead}. 
However, no generalisations of the P\'osa or Chv\'atal conditions for simple graphs were known except for a
P\'osa-type codegree condition for $3$-graphs proved by Sch\"ulke in \cite{Schuelke}.
The present work provides a Chv\'atal-type codegree condition for every fixed uniformity.
In particular, it improves the Dirac-type result by R\"odl, Ruci\'nski, and Szemer\'edi in \cite{RRSApprox} and confirms, as a corollary, a P\'osa-type condition conjectured by Sch\"ulke for every uniformity in \cite{Schuelke}.

To state the main theorem, let $G$ be a $k$-graph. For every $A\in\binom{V(G)}{k-2}$, 
we define
\[c_1(A)\le c_2(A)\le\cdots\le c_{\abs{V(G)}-k+2}(A)\]
to be the non-decreasing arrangement
of the multiset
$\mathcal{R}(A) := \{ d(A\cup\{x\}) : x\in V(G) \setminus A\}$.
Equivalently,  $c_1(A),\dots,c_{\abs{V(G)}-k+2}(A)$ form the ordinary degree sequence of the link graph $L(A)$. 
The \emph{link graph} $L(A)$ of $A$ is defined to be the graph on $V(G)\setminus A$ in which
$xy$ is an edge precisely when $A\cup\{x,y\}\in E(G)$. 

We are now ready to state the main theorem.

\begin{theorem}[Main]\label{thm:main}
Fix $k\ge 3$ and $0 < \alpha < 1/4$.  There exists
$n_0$ such that for every $n\ge n_0$ the following holds.
Let $G$ be a $k$-graph on the vertex set $[n]$.  Suppose, for
every $A\in\binom{[n]}{k-2}$ and every integer
$1\le i<\min\{\min A,(n-k+2)/2\}$,
at least one of the following holds:
\begin{lefttagged}
\begin{align*}
 c_i(A)&\ge i+\alpha n,
 \tag{$C_1$}\label{c1}\\
 c_{n-k+3-i-\lceil\alpha n\rceil}(A)
 &\ge n-k+2-i.
 \tag{$C_2$}\label{c2}
\end{align*}
\end{lefttagged}
Then, $G$ contains a (tight) Hamiltonian cycle.
\end{theorem}
\noindent
Equivalently, \eqref{c2} says that at least $i + \lceil\alpha n\rceil$ vertices
$x\in V(G)\setminus A$ satisfy
$d(A\cup\{x\}) \ge n-k+2-i$.

In \cref{sec:sharp}, we show the conditions in \cref{thm:main} are sharp in a certain sense.
\cref{thm:main}, in addition,  implies the following P\'osa-type result.

\begin{corollary}[P\'osa-type]\label[corollary]{cor:posa}
Fix $k\ge 3$ and $0 < \alpha < 1/4$.  There exists
$n_0$ such that for every $n\ge n_0$ the following holds.
Let $G$ be a $k$-graph on the vertex set $[n]$.
Suppose, for
every $S\in\binom{[n]}{k-1}$,
\begin{lefttagged}
\begin{align*}
d(S)&\ge \min\{\min S,(n-k+2)/2\}+\alpha n.
\tag{$P$}\label{p}
\end{align*}
\end{lefttagged}
Then, $G$ contains a (tight) Hamiltonian cycle.
\end{corollary}

\begin{proof}
Fix $A\in\binom{[n]}{k-2}$ and $1\le i<\min\{\min A,(n-k+2)/2\}$. 
If $c_i(A) < i + \alpha n$, then there are at most $n-k+2-i$ elements in $\mathcal{R}(A)$ that are at least $i + \alpha n$. However, condition \eqref{p} implies that $d(A\cup\{j\}) \ge i + \alpha n$ for every $j \in \set{i, \dots, n} \setminus A$, whereas $|\set{i, \dots, n} \setminus A| = n - k + 3 - i > n - k + 2 - i$, a contradiction.
Hence, \eqref{c1} is always satisfied.
\end{proof}

We remark that the condition \eqref{p} in \cref{cor:posa} can be made slightly sharper by replacing $(n-k+2)/2$ with $\ceil{(n-k+2)/2} - 1$ (its proof remains the same). 
For $k=3$, \cref{cor:posa} recovers Sch\"ulke's theorem
\cite[Theorem~1.4]{Schuelke}.  For arbitrary fixed $k$, it confirms the P\'osa-type extension conjectured in Section~7 of the same paper.
In \cref{sec:strict}, we show \cref{thm:main} is strictly stronger than \cref{cor:posa} by examples which always satisfy \eqref{c1} or \eqref{c2} but never \eqref{p} (no matter how one orders the vertices).

\cref{thm:main} extends the sufficient-condition aspect of Chv\'atal's theorem. Its characterisation aspect behaves differently.
A function $D:\binom{[n]}{k-1} \rightarrow \set{0,\dots,n-k+1}$ is 
\emph{Hamiltonian} if every $k$-graph $G$ on $[n]$ satisfying $d_G(S)\ge D(S)$ for every $S\in\binom{[n]}{k-1}$ contains a Hamiltonian cycle.
We also call a Hamiltonian function a \emph{Hamiltonian sequence} for $k = 2$ and a \emph{Hamiltonian matrix} for $k=3$.
For graphs, Chv\'atal's criterion in \cite{Chvatal} in fact provides an equivalent characterisation of all Hamiltonian sequences in the following sense:
if $D(i) \le D(j)$ whenever $i \le j$, then $D$ is a Hamiltonian sequence if and only if, for every $i < n/2$, we have $D(i) \ge i + 1$ or $D(n-i) \ge n-i$.
Sch\"ulke suggested in \cite{Schuelke} that a complete characterisation of all Hamiltonian matrices is very desirable.
However, we will give a pessimistic answer to this question by proving that, for $k \ge 3$, determining whether a function is Hamiltonian is NP-hard; see \cref{thm:hard}.
That is to say, unlike the $k=2$ case, we cannot expect an equivalent characterisation checkable in P  unless P $=$ NP. In fact, in \cref{sec:hard}, we prove \cref{thm:alpha-hard}, which is a stronger version of \cref{thm:hard}.
\begin{theorem}[Hardness]\label{thm:hard}
For any fixed $k \ge 3$, determining whether a function $D:\binom{[n]}{k-1} \rightarrow \set{0,\dots,n-k+1}$ is 
Hamiltonian is NP-hard.   
\end{theorem}

%We show that, for every fixed uniformity at least three, recognising Hamiltonian codegree bounds is NP-hard, already for functions taking only two values. Thus a complete characterisation with a polynomial-time test would imply $\mathrm{P}=\mathrm{NP}$.

%
%Chv\'atal's classical theorem in \cite{Chvatal} in fact provides an equivalent characterisation of all Hamiltonian sequences in the following sense:
%
%If $D(i) \le D(j)$ whenever $i \le j$, then $D$ is a Hamiltonian sequence if and only if, for every $i < n/2$, we have $D(i) \ge i + 1$ or $D(n-i) \ge n-i$.
%

We survey some established degree conditions for other spanning tight structures in hypergraphs. 
Pavez-Sign\'e,
Sanhueza-Matamala, and Stein proved an asymptotically sharp
minimum codegree condition for every bounded-degree spanning tight tree
\cite{PSSHypertrees}; they also obtained a minimum codegree theorem for powers
of tight Hamilton cycles and, more generally, for bounded-tree-width
spanning hypergraphs \cite{PSSPosaSeymour}.  More recently, Di Braccio,
Hearn, Lada, Neve, and the author determined the exact minimum codegree
threshold for a spanning tight component in a $4$-graph
\cite{DiBraccioEtAl}.
Bowtell and Hyde proved P\'osa-type vertex-degree conditions for
perfect matchings in $3$-graphs \cite{BowtellHyde}.
To the best of our knowledge, no Chv\'atal-type results appear to be known for other spanning structures in hypergraphs.

\subsection*{Proof strategy}
The proof of \cref{thm:main} uses the method of Hamilton frameworks by Lang and
Sanhueza-Matamala in \cite{LSMbandwidth}.
The notion of a Hamilton framework was initially introduced by the same authors in their study of minimum-degree conditions for (tight)
Hamilton cycles \cite{LSMminimum}, 
and the Hamilton-framework method was further developed in \cite{LSMspanning, LSMbandwidth}.

Roughly, under this method, we need to verify the following: first, the conditions \eqref{c1} and \eqref{c2} satisfy the four Hamilton framework axioms defined in \cite[Definition~3.1]{LSMbandwidth}; second, these conditions are inherited by many subgraphs of a host hypergraph. 

\subsection*{Organisation}

In \cref{sec:2}, we introduce the preliminaries and the Hamilton-framework machinery.  
In \cref{sec:proof}, we prove the main theorem.
In \cref{sec:strict}, we prove \cref{thm:main} is stronger than \cref{cor:posa}.
In \cref{sec:sharp}, we discuss the sharpness of the conditions in \cref{thm:main} and conjecture an exact version of it without $\alpha n$.
In \cref{sec:hard}, we prove that determining whether a function is Hamiltonian is NP-hard for every uniformity $k\ge3$.

\section{Preliminaries and Hamilton framework}\label{sec:2}

For a $k$-graph $G$, its \emph{line graph} has vertex set $E(G)$ with two hyperedges adjacent when they meet in exactly $k-1$ vertices.  
A subgraph of $G$ is \emph{connected} if it has no isolated vertices and its
edges induce a connected subgraph of $G$'s line graph.
A \emph{component} is an
edge-maximal connected subgraph.
A (tight) \emph{walk} is a vertex sequence,
with repetitions permitted, in which every $k$ consecutive terms form
an edge.
It is \emph{closed} when the same holds with cyclic indexing.
The \emph{order} of the walk is the length of the sequence.
A \emph{fractional matching} of $G$ is a function
$\mu:E(G)\to[0,1]$ such that
$\sum_{e\ni v}\mu(e)\le1$ for every $v\in V(G)$, and its size is
$\sum_{e\in E(G)}\mu(e)$. 
We say $\mu$ is \emph{perfect} if its size is $|V(G)|/k$.
We follow the convention that 
$\min \emptyset = \infty$.

We use the following definition from
\cite[Definition~3.1]{LSMbandwidth}.

\begin{definition}[Hamilton framework \cite{LSMbandwidth}]\label[definition]{def:hamilton-framework}
Let $\cP$ be a family of $s$-vertex $k$-graphs.  The family $\cP$ admits a
\emph{Hamilton framework $F$} if $F$ assigns to every $H\in\cP$ an
$s$-vertex subgraph $F(H)\subseteq H$ such that:
\begin{enumerate}[label=\textup{(F\arabic*)}]
\item $F(H)$ is a component  (connectivity);
\item $F(H)$ has a perfect fractional matching (space);
\item \label{f3} $F(H)$ contains a closed walk whose order is congruent to $1$ modulo $k$ (aperiodicity);
\item $F(H) \cup F(H')$ is connected whenever $H, H' \in \cP$ are obtained by deleting distinct vertices from the same $(s+1)$-vertex $k$-graph (consistency).
\end{enumerate}
\end{definition}

The notion of property graphs was introduced by Lang \cite[Definition~2.6]{LangTiling} in the context of perfect tilings.

\begin{definition}[Property graph \cite{LangTiling}]
Let $G$ be an $n$-vertex $k$-graph and $\cP$ be a family of $s$-vertex $k$-graphs with $s \le n$.
Define the \emph{property $s$-graph}
$P^{(s)}(G,\cP)$ on $V(G)$ by
\[
 E(P^{(s)}(G,\cP))=\{S\in\tbinom{V(G)}s:G[S]\in\cP\}.
\]
\end{definition}

We use the following specialised consequence of
\cite[Theorem~3.5]{LSMbandwidth}. Basically, the following theorem says if $G$ satisfies the property $\mathcal{P}$, which admits a Hamilton framework, in a robust local way, then $G$ contains a Hamiltonian cycle.
\begin{theorem}[Hamilton-framework theorem \cite{LSMbandwidth}]\label{thm:framework}
Fix $k\ge 2$ and $s\ge 2k+1$.  Let $\cP$ be a family of $s$-vertex $k$-graphs
that admits a Hamilton framework.  There exists $n_0$ such that
every $n$-vertex $k$-graph $G$ with $n\ge n_0$ satisfying
\begin{equation}
 \delta_{2k}(P^{(s)}(G,\cP))
 \ge(1-s^{-2})\binom{n-2k}{s-2k},
 \label{eq:robust}
\end{equation}
contains a (tight) Hamiltonian cycle.
\end{theorem}

We define $\cP(s)$ to be the family of $s$-vertex $k$-graphs $H$ for which there is a labelling $V(H) = [s]$ such that
for
every $A\in\binom{[s]}{k-2}$ and every integer
$1\le i<\min\{\min A,(s-k+2)/2\}$,
at least one of the following holds:
\begin{lefttagged}
\begin{align*}
 c_i(A)&\ge i+ 1,
 \tag{$D_1$}\label{d1}\\
 c_{s-k+2-i}(A)
 &\ge s-k+2-i.
 \tag{$D_2$}\label{d2}
\end{align*}
\end{lefttagged}
Equivalently, \eqref{d2} says at least $i + 1$ vertices $x\in V(H)\setminus A$ satisfy 
$d(A\cup\{x\}) \ge s-k+2-i$. 

We shall see, for every sufficiently large $s$ such that $s-k+2$ is odd, $\cP(s)$ admits a Hamilton framework by simply taking $F(H) = H$ for every $H \in \cP(s)$. 
The oddness condition on $s-k+2$ is imposed only to simplify the proof of \ref{f3}, namely, that $H$ contains a closed walk whose order is congruent to $1$ modulo $k$, but we believe that, even without the oddness condition, such an $H$ will still satisfy \ref{f3} when $k\ge3$.

\section{Proof of the main theorem}\label{sec:proof}

Throughout this section we fix $k\ge3$.
For convenience, let $U:=\{\lceil (s-k+2)/2\rceil,\ldots,s\}$ denote the upper half of $[s]$.

\cref{lem:connected} to \cref{lem:aperiodic} establish the four framework axioms for $\cP(s)$.

\begin{lemma}[Connectivity]\label[lemma]{lem:connected}
For every sufficiently large $s$ and every $H\in\cP(s)$,
every $(k-1)$-subset of $V(H)$ lies in the same component. In particular, $H$ is a component.
\end{lemma}

\begin{proof}
For every $A\in\binom U{k-2}$, the condition of \cref{thm:classic} is satisfied by $L(A)$, and therefore $L(A)$ is connected. 
In particular, every $(k-1)$-tuple containing $A$ belongs to the same component.
Namely, for $S,S'\in\binom {[s]}{k-1}$, if $S \cap S' \in \binom{U}{k-2}$, then they belong to the same component. Since for every $S,S'\in\binom {U}{k-1}$, there is a sequence of $(k-1)$-tuples $S_0, \dots, S_{\ell}$ in $\binom {U}{k-1}$ with $S_0 = S$ and $S_{\ell} = S'$ such that $|S_{i-1} \cap S_{i}| = k-2$, $S$ and $S'$ belong to the same component. We denote this component by $\cC$.

We now prove by descending induction on $i$ that every $(k-1)$-tuple $S$ with $\min S = i$ lies in $\cC$.  
The preceding paragraph proves the assertion for $i \ge \lfloor (s-k+2)/2 \rfloor$.
For $1\le i\le\lfloor (s-k+2)/2 \rfloor - 1$, let $A = S \setminus \set{i}$, and let $I$ be the component of $i$ in $L(A)$.  
\textbf{Case I}: $I$ contains some $j>i$. In this case, a path from $i$ to $j$ in $L(A)$ implies a tight walk from $S = A \cup \set{i}$ to $S' := A \cup \set{j}$. As $S'$ is contained in $\cC$ by the induction hypothesis, so is $S$.
\textbf{Case II}:
$I\subseteq[i]$. Let $i_0 = \abs{I}$. Then, $i_0 \le i < \min\set{\min A, (s-k+2)/2}$. Since at least $i_0$ vertices in $L(A)$ (the vertices in $I$) have degree at most $i_0-1$, we have $c_{i_0}(A) \le i_0 - 1$. As \eqref{d1} fails at $i_0$, \eqref{d2} implies, in particular, there is a vertex in $L(A)$ with degree at least $\abs{L(A)} - i_0$, but this is impossible as every vertex in $L(A)$ has degree at most $\abs{L(A)} - i_0 - 1$. Hence, this case is not possible.

Therefore, every $(k-1)$-tuple $S \subseteq V(H)$ lies in $\cC$.
\end{proof}

\begin{lemma}[Consistency]\label[lemma]{lem:consistency}
For every sufficiently large $s$,
$H \cup H'$ is connected whenever $H, H' \in \cP(s)$ are obtained by deleting distinct vertices from the same $(s+1)$-vertex $k$-graph.
\end{lemma}

\begin{proof}
This follows immediately from \cref{lem:connected}: the component $\cC$ that contains a $(k-1)$-subset of $V(H) \cap V(H')$ is the unique component of $H \cup H'$.
\end{proof}

A \emph{fractional vertex cover} is a function
$w:V(H)\to[0,1]$ such that $\sum_{v\in e}w(v)\ge1$ for every
$e\in E(H)$; its size is $\sum_{v\in V(H)}w(v)$.
Strong duality of linear programming states that the maximum size of a fractional
matching always equals the minimum size of a fractional vertex cover.

\begin{lemma}[Space]\label[lemma]{lem:space}
For every sufficiently large $s$,
every $H\in\cP(s)$ has a perfect fractional matching.
\end{lemma}

\begin{proof}
Let $w$ be an arbitrary fractional vertex cover for $H$, and let $v_1, \dots, v_s$ be an ordering of $V(H) = [s]$ such that $w(v_1) \le \dots \le w(v_s)$.
Let $x(v_i) = w(v_i) - 1/k$ for $i \in [s]$. 
Then, $x(v_1) \le \dots \le x(v_s)$.
As $w$ is a vertex cover, for every $e \in E(H)$, we have
\begin{equation}
\sum_{v \in e}x(v) = \sum_{v \in e}(w(v) - 1/k) \ge 1 -1 = 0.
\label{eq:space}
\end{equation}
As a fractional matching is perfect if and only if it has size at least $s/k$,
by strong duality of linear programming, it suffices to prove 
$\sum_{v \in V(H)}x(v) = \sum_{v \in V(H)}(w(v) - 1/k) \ge s/k - s/k = 0$.

For an integer $j$, define
$f(j)=s - j +k-2$.
Then, when $k-1 \le j < (s+k-2)/2$, $j < f(j) \le s - 1$. In particular, when $j = \ceil{(s+k-2)/2} - 1$, $f(j) = \floor{(s+k-2)/2} + 1 \ge j + 1$; and when $j = k-1$, $f(j) = s -1$.
We have the following result.

\begin{claimblock}
\begin{nestedclaim}\label[nestedclaim]{claim:space}
For $k-1 \le j < (s+k-2)/2$, we have
$x(v_{j-k+2})+\dots+x(v_j)+x(v_{f(j)}) \ge 0$.
\end{nestedclaim}

\begin{proof}[Proof of \cref{claim:space}]
Let $k-1 \le j < (s+k-2)/2$ and $J = \set{v_i : i\in[j]}$. 
Let $A$ be set of the $k-2$ largest elements in $J$ (with respect to the original vertex labels in $[s]$) and $J_o = J \setminus A$ (with $\abs{J_o} = j - k + 2$).
By the definition of $A$,
$\min A \ge j - k + 3$. 
We also observe that, for every $w \in J_o$, we have
\begin{equation}
x(v_{j-k+2}) + \dots + x(v_j)
\ge
\sum_{v\in\set{w}\cup A} x(v).
\label{eq:space2}
\end{equation}
Set $i = j - k + 2 = \abs{J_o}$. Then
$i < \min\set{\min A, (s-k+2)/2}$.

If \eqref{d1} holds at $i$, then $c_i(A) \ge i + 1$, which implies there are at most $i-1$ vertices $v$ such that $d(\set{v} \cup A) \le i$. As $\abs{J_o} = i$, there is $w \in J_o$ such that 
$d(\set{w} \cup A) \ge i + 1$.
Hence, there is a vertex $u$ with $x(u) \le x(v_{s-i}) = x(v_{f(j)})$ such that $\set{u, w} \cup A$ is an edge. 
Then, \eqref{eq:space} and \eqref{eq:space2}
imply that 
\[
x(v_{j-k+2})+\dots+x(v_j)+x(v_{f(j)}) 
\ge
\sum_{v\in\set{w}\cup A} x(v) 
+ x(u)
\ge 0.
\]

If \eqref{d2} holds at $i$, then there are at least $i + 1$ vertices $v$ such that $d(\set{v} \cup A) \ge s - k + 2 - i$.
In particular, there is a vertex $u$ with $x(u) \le x(v_{s-i}) = x(v_{f(j)})$ such that
$d(\set{u} \cup A) \ge s - k + 2 - i = \abs{L(A)} - \abs{J_o}$.
If $u \notin J_o$, there is a vertex $w \in J_o$ such that $\set{u, w} \cup A$ is an edge.
By \eqref{eq:space} and \eqref{eq:space2}, we have
\[
x(v_{j-k+2})+\dots+x(v_j)+x(v_{f(j)}) 
\ge
\sum_{v\in\set{w}\cup A} x(v) 
+ x(u)
\ge 0.
\]
If $u \in J_o$, then, since
$d(\set{u} \cup A) \ge s - k + 2 - i \ge i + 1$,
there is a vertex $w$ with $x(w) \le x(v_{s-i}) = x(v_{f(j)})$ such that $\set{w, u} \cup A$ is an edge.
By \eqref{eq:space} and \eqref{eq:space2}, we have
\[
x(v_{j-k+2})+\dots+x(v_j)+x(v_{f(j)}) 
\ge
\sum_{v\in\set{u}\cup A} x(v) 
+ x(w)
\ge 0.
\]
\end{proof}
\end{claimblock}

If $x(v_1) \ge 0$, we have $\sum_{v \in V(H)}x(v) \ge 0$ immediately.
Otherwise, let $\ell = \max\set{i\in[s]:x(v_i) < 0}$. Then, we have $\ell \le (s+k-2)/2$.
Indeed, if $\ell \ge \floor{(s+k-2)/2} + 1$, then 
$x(v_{\floor{(s+k-2)/2} - k + 2}) + \dots + x(v_{\floor{(s+k-2)/2} + 1}) < 0$, which contradicts \cref{claim:space} with $j = \ceil{(s+k-2)/2} - 1$.
If $\ell < k-1$, \cref{claim:space} immediately implies
\[
\sum_{i\in[s]} x(v_i)
\ge 
x(v_1) + \dots + x(v_{k-1}) + x(v_{f(k-1)})
\ge 0.
\]
For the remainder of the proof, assume that $s+k$ is odd. When $s+k$ is even, the proof is almost identical and uses only the additional fact that $x(v_\ell) + x(v_s) \ge 0$ when $\ell$ is exactly $(s+k-2)/2$.
Now as $k-1 \le \ell < (s+k-2)/2$, we again have
\begin{align*}
\sum_{i\in[s]} x(v_i)
&\ge 
x(v_1) + \dots + x(v_\ell) + x(v_{f(\ell)}) + \dots + x(v_{f(k-1)})  \\
&\ge 
\sum_{k-1 \le j \le \ell} 
\Bigl(
x(v_{j-k+2})+\dots+x(v_j)+x(v_{f(j)}) 
\Bigr) \\
&\ge 
0.
\end{align*}
\end{proof}

An \emph{ordered edge} is an ordering of the vertices of
an edge. For example, if $\set{v_1, \dots, v_k}$ is an edge, then $e := (v_1, \dots, v_k)$ is an ordered edge. The parity assumption on $s-k+2$ enters only here.

\begin{lemma}[Aperiodicity]\label[lemma]{lem:aperiodic}
For every sufficiently large $s$ such that $s-k+2$ is odd, every $H\in\cP(s)$ contains a closed
walk of order $1$ modulo $k$.
\end{lemma}

\begin{proof}
We first observe that if $e= (v_1,\dots,v_k)$ and $e'=(v_1,\dots,v_{r-1},w,v_{r+1},\dots,v_k)$ are ordered edges, then their 
\emph{concatenation}
\[
ee' :=
v_1, \dots, v_k, v_1,\dots,v_{r-1},w,v_{r+1},\dots,v_k
\]
is a tight walk of order $0$ modulo $k$.

We denote $m = s - k + 2$, and recall $U=\{\lceil m/2\rceil,\ldots,s\}$.
We next prove the following claim.
\begin{claimblock}
\begin{nestedclaim}\label[nestedclaim]{claim:aperiodic-swap}
For every $A\in\binom{U}{k-2}$ and every $xy\in E(L(A))$, there is a
tight walk of order $0$ modulo $k$ from any ordering
of $A\cup\set{x,y}$ to the ordering obtained by interchanging $x$ and
$y$ and leaving every element of $A$ in its original coordinate.
\end{nestedclaim}

\begin{proof}[Proof of \cref{claim:aperiodic-swap}]

Fix $A \in \binom{U}{k-2}$, $xy\in E(L(A))$ and an ordering $f_1$ of $A \cup \set{x,y}$. We assume, in the ordered edge $f_1$, the coordinates of $x,y$ are $o_1, o_2$ respectively.

\cref{thm:classic} applied to $L(A)$ implies that $L(A)$ has a Hamiltonian cycle, say
$v_1 v_2\dots v_m$.
Let $e_1$ be the ordering $f_1$ but with $v_1, v_2$ in the places of $x,y$ respectively.
For $2 \le i \le m+1$, we define, recursively, $e_{i}$ to be the ordered edge obtained by replacing $v_{i-1}$ with $v_{i+1}$ in $e_{i-1}$, where subscripts on the $v_i$ are taken modulo $m$.
In particular, the underlying (unordered) edge of $e_{m+1}$ is again $A\cup\set{v_1,v_2}$, the same as the underlying edge of $e_1$.

Since $m$ is odd, 
the vertex order in $e_{m+1}$ is the same as $e_1$
except with $v_1$ and $v_2$ interchanged. Indeed, in $e_1$, the coordinates of $v_1$ and $v_2$ are $o_1$ and $o_2$, respectively. Then inductively, for $i \in [m+1]$, the coordinates of $v_{i}$ and $v_{i+1}$ are $o_1$ and $o_2$ if $i$ is odd and $o_2$ and $o_1$ if $i$ is even. $m+1$ is even, and $v_{m+1} = v_1$ and $v_{m+2} = v_2$.

In addition, $e_1 \dots e_{m+1}$ is a tight walk in $H$ of order $0$ modulo $k$.

As $L(A)$ is connected, there is a walk $w_1 w_2 \dots w_{t-1} w_t$ in $L(A)$ with $w_1 = x, w_2 =y, w_{t-1}=v_1, w_t=v_2$.
For $2 \le i \le t-1$, we define $f_{i}$ to be the ordered edge obtained by replacing $w_{i-1}$ with $w_{i+1}$ in $f_{i-1}$.
Then, $f_1 \dots f_{t-1}$ is a tight walk in $H$ of order $0$ modulo $k$, and $f_{t-1}$ is either $e_1$ or $e_{m+1}$.

For $i \in [t-1]$, we define $f_i'$ to be $f_i$ with $w_i$ and $w_{i+1}$ interchanged. 
Then, $f_{t-1}' \dots f_1'$ is a tight walk in $H$ of order $0$ modulo $k$, and $f_{t-1}'$ is  $e_{m+1}$ or $e_1$ when $f_{t-1}$ is $e_1$ or $e_{m+1}$, respectively. Assume without loss of generality that $f_{t-1} = e_1$.

Therefore, 
$f_1 \dots f_{t-1}  e_2 \dots e_{m} f_{t-1}' \dots f_1'$ is the desired tight walk in $H$.
\end{proof}
\end{claimblock}

Choose $A\in\binom{U}{k-2}$ and a Hamiltonian cycle in $L(A)$ as in the
proof of \cref{claim:aperiodic-swap}. Since $\abs{U\setminus A} = (m+1)/2 > m/2$,
two consecutive vertices $x,y$ of this Hamiltonian cycle lie in
$U\setminus A$.
Then, $e:=A\cup\set{x,y}\in E(H)$ and $e\subseteq U$.

We write $e=\set{v_1,\dots,v_k}$,
and for distinct $i,j\in[k]$, we denote
$A_{ij}:=e\setminus\set{v_i,v_j}$.
Then $A_{ij}\in\binom{U}{k-2}$ and
$v_iv_j\in E(L(A_{ij}))$. 
By \cref{claim:aperiodic-swap}, in any ordering of $e$, any two vertices of $e$ can be interchanged in their coordinates by a tight walk of order $0$ modulo $k$.
Starting with $(v_1,\dots,v_k)$ and successively interchanging $v_k$ with $v_{k-1},v_{k-2},\dots,v_1$ results in a tight walk terminating at $(v_k,v_1,\dots,v_{k-1})$. 
Since each interchange appends a multiple of $k$ vertices to the tight walk,
the final tight walk we obtain again has order $0$ modulo $k$. 

Appending one further vertex $v_k$ to the end yields a closed tight walk of order $1$ modulo $k$, as desired. 
\end{proof}

We next show that the global condition from \cref{thm:main} is inherited by many induced subgraphs on $s$ vertices.
We use the following form of Hoeffding's inequality.

\begin{lemma}[Hoeffding \cite{Hoeffding}]\label[lemma]{lem:hoeffding}
Let $a_1,\ldots,a_N\in[0,1]$, let $1\le m\le N$, and let $I$ be a uniformly random $m$-subset of $[N]$. 
Define $X=\sum_{i\in I}a_i$. 
Then for every $t>0$,
\[
\Pr\bigl(\abs{X-\mathbb E X}\ge t\bigr)
\le 2\exp\left(-\frac{2t^2}{m}\right).
\]
\end{lemma}

\begin{lemma}[Inheritance]\label[lemma]{lem:inheritance}
Fix $0<\alpha<1/4$. There exists $s_0$ such that, for every integer $s\ge s_0$, there exists $n_0$ for which the following holds whenever $n\ge n_0$. 
Let $G$ be a $k$-graph on $[n]$ satisfying the hypotheses of \cref{thm:main}. 
We have
\[
\delta_{2k} \bigl(P^{(s)}(G,\cP(s))\bigr)
\ge(1-s^{-2})\binom{n-2k}{s-2k}.
\]
\end{lemma}

\begin{proof}
Denote $\rho = s/n$ and $\beta = \alpha/100$.
Let $Q \subseteq T \subseteq [n]$ with $|Q|=2k$ and $|T| \le 3k-1$.
Let $S$ be a uniformly random $s$-set containing $T$.
For $v \in S$, we define the order of $v$ in $S$ to be $o_S(v):=\abs{S\cap[v]}$.
For any $(k-2)$-subset $A\subseteq S$,
define $m_S(A):=\min_{v\in A}o_S(v)$.
We omit the subscript $S$ when it is clear.

For any $X \subseteq [n]$, we have
$\mathbb E(\abs{S\cap X})
= \abs{T\cap X} + \frac{s-\abs T}{n-\abs T}\abs{X\setminus T}$.
An elementary calculation gives
\[
\big|
\mathbb E(|S\cap X|)-\rho |X|
\big|
\le |T| \le 3k-1\le\beta s/2.
\]
Consequently, if $\abs{S\cap X}$ differs from $\rho\abs X$ by more than $\beta s$, then it differs from its mean by more than $\beta s/2$. 
Hence, by \cref{lem:hoeffding},
\begin{equation}\label{eq:inheritance}
\Pr(\big|
\abs{S\cap X}-\rho\abs X
\big| 
>\beta s)
\le 
2\exp\left(-\frac{\beta^2s^2}{2(s-\abs T)}\right)
\le 
2\exp(-\beta^2s/2).    
\end{equation}

For every integer $1\le i<(s-k+2)/2$, define 
$
i^{+}:= \max
\left\{
1, \floor{\frac{i-10\beta s}{\rho}}
\right\}
$.
For a uniformly random $s$-set $S$ containing $Q$,
we call $S$ \emph{good} if the following hold and \emph{bad} otherwise.
\begin{enumerate}[label=\textup{(\roman*)}]
\item For every integer $1\le i<(s-k+2)/2$,
\[
\big|
\abs{S\cap[i^{+}]}-\rho i^{+}
\big| 
\le \beta s.
\]
\item For every $A\in\binom S{k-2}$ and $x\in S\setminus A$,
\[
\big|
d_{G[S]}(A\cup\set{x})-\rho d_G(A\cup\set{x})
\big|
\le\beta s.
\]
\item For every $A\in\binom S{k-2}$ and every integer $1\le i<(s-k+2)/2$,
\[
\big|
\abs{S\cap X}-\rho\abs X
\big| 
\le\beta s
\]
for both $X=\set{x\in[n]\setminus A:d_G(A\cup\set{x})<i^{+}+\alpha n}$ and $X=\set{x\in[n]\setminus A:d_G(A\cup\set{x})\ge n-k+2-i^{+}}$.
\end{enumerate}
Applying \eqref{eq:inheritance} with $T=Q$ for \textup{(i)}, $T=Q\cup A\cup\set{x}$ for \textup{(ii)}, and $T=Q\cup A$ for \textup{(iii)}, and the union bound yields
\[
\begin{aligned}
\Pr(S\text{ is bad})
&
\le2\exp(-\beta^2s/2)
\left(
s + 
\sum_{A\cup\set{x}\in\binom{[n]}{k-1}}
\Pr(A\cup\set{x}\subseteq S)
+ 2s\sum_{A\in\binom{[n]}{k-2}}\Pr(A\subseteq S)
\right)\\
&
=2\exp(-\beta^2s/2)
\left(
s+\binom{s}{k-1}+2s\binom{s}{k-2}
\right)\\
&<s^{-2}
\end{aligned}
\]
Hence, 
$\Pr(S\text{ is good}) \ge 1-s^{-2}$. It then remains to show that $G[S]\in\cP(s)$ for every good $S$.

Fix a good $S$, an $A\in\binom S{k-2}$, and an integer $i$ satisfying
$1\le i< \min\set{m_S(A),(s-k+2)/2}$.
Recall $i^{+}= \max\left\{1,\floor{\frac{i-10\beta s}{\rho}}\right\}$.
We would like to show that
$1 \le i^{+}<\min\set{\min A, (n-k+2)/2}$
so that the hypotheses of \cref{thm:main} apply to $A$ and $i^{+}$.
If $m(A) = 1$, the desired conclusion is vacuously true. 
Hence, we assume $m(A) \ge 2$, which implies $\min A \ge 2$.
If $i^{+}=1$, the desired result obviously follows.
If $i^{+}>1$,  
$\abs{S\cap[i^{+}]}\le\rho i^{+}+\beta s<i$. If $i^{+}\ge\min A$, then $\abs{S\cap[i^{+}]}\ge m(A)>i$, a contradiction with $i<m(A)$. 
Hence, $i^{+}<\min A$. Also, $\rho i^{+}<i<(s-k+2)/2\le\rho(n-k+2)/2$, so $i^{+}<(n-k+2)/2$.

Suppose first that \eqref{c1} holds at $i^{+}$. 
There are at most $i^{+}-1$ vertices $x\in[n]\setminus A$ with $d_G(A\cup\set{x})<i^{+}+\alpha n$. 
If $i^{+}>1$, at most
$\rho(i^{+}-1)+\beta s<i$
of them lie in $S$, and if $i^{+}=1$, there are none. 
In either case, there are fewer than $i$ such vertices in $S$.
Every other $x\in S\setminus A$ satisfies
$
d_{G[S]}(A\cup\set{x})
\ge \rho d_G(A\cup\set{x})-\beta s
\ge \rho (i^{+}+\alpha n)-\beta s
>i+\beta s
$.
Hence, \eqref{d1} holds at $i$.

Suppose instead that \eqref{c2} holds at $i^{+}$. At least $i^{+}+\alpha n$ vertices $x\in[n]\setminus A$ satisfy $d_G(A\cup\set{x})\ge n-k+2-i^{+}$, and at least
$\rho(i^{+}+\alpha n)-\beta s>i+\beta s$
of them are in $S$.
If $i^{+}=1$, $A \cup\set{x}$ is universal in $G$ for each such $x$, and thus $d_{G[S]}(A\cup\set{x})=s-k+1$.
If $i^{+}>1$, then each such $x$ satisfies
$d_{G[S]}(A\cup\set{x})
\ge \rho d_{G}(A\cup\set{x})-\beta s
\ge \rho (n-k+2-i^{+})-\beta s
\ge s-k+2 - i + \beta s$.
In either case, 
at least $i+\beta s$ vertices $x \in S\setminus A$ satisfy
$d_{G[S]}(A\cup\set{x}) \ge s-k+2-i$. 
Hence, \eqref{d2} holds at $i$.

Therefore, $G[S]\in\cP(s)$ as desired.
\end{proof}

Now we have all the ingredients for the proof of the main theorem.

\begin{proof}[Proof of \cref{thm:main}]
Fix $k\ge3$ and $0<\alpha<1/4$. Choose $s$ sufficiently large such that $s-k+2$ is odd, and then choose $n$ sufficiently large.
By \cref{lem:connected,lem:space,lem:aperiodic,lem:consistency}, the family $\cP(s)$ admits a Hamilton framework, with $F(H)=H$ for every $H\in\cP(s)$.

Let $G$ be an $n$-vertex
$k$-graph satisfying the hypotheses of \cref{thm:main}. By
\cref{lem:inheritance},
$\delta_{2k}\bigl(P^{(s)}(G,\cP(s))\bigr)
\ge(1-s^{-2})\binom{n-2k}{s-2k}$.
Therefore, \cref{thm:framework} applies and yields a Hamiltonian cycle in $G$.
\end{proof}

\section{Further discussion}\label{sec:final}

\subsection{Comparison between
\texorpdfstring{\cref{thm:main}}{the Chvatal-type theorem} and
\texorpdfstring{\cref{cor:posa}}{the Posa-type theorem}}\label{sec:strict}
\ \\

In this part, we show \cref{thm:main} is strictly stronger than \cref{cor:posa}.

\begin{example}
\label[example]{ex:strict}
Fix $k\ge3$ and $0<\alpha<1/(10k)$, and let $n$ be sufficiently
large. Let $t = 2\ceil{\alpha n}$.
Choose pairwise disjoint $(k-1)$-sets $S_1,\ldots,S_t$. Starting with
the complete $k$-graph on $[n]$, retain exactly $1 + \ceil{\alpha n}$ edges containing each $S_j$ and delete all other edges containing $S_j$. \end{example}
By our construction, $d(S_j) = 1+\ceil{\alpha n}$ for each $j \in [t]$. For any $(k-1)$-set $T$ distinct from all the $S_j$, $d(T) \ge n-k+1-t$.
Indeed, if a deleted edge $e$ contains $T$, then $e = T \cup S_j$ for some $j \in [t]$.
Therefore, there are at most $t$ deleted edges containing $T$.

\begin{proposition}\label[proposition]{prop:strict}
Let $G$ be the $k$-graph defined in \cref{ex:strict}.
Under every ordering of $V(G)$,
\eqref{c1} holds for every $A\in\binom{[n]}{k-2}$ and every integer
$1\le r< \min\{\min A,(n-k+2)/2\}$, but no ordering satisfies an even weaker version of \eqref{p}: 
$d(S)\ge \min\{\min S,(n-k+2)/2\}+ 1$ for every $S \in \binom{V(G)}{k-1}$.
\end{proposition}

\begin{proof}
For any $(k-2)$-set $A$, there is at most one vertex $x$ such that $A \cup \set{x}$ is some $S_j$. Therefore, 
$c_1(A) \ge 1 + \alpha n$,
and for $2 \le i <\min\{\min A,(n-k+2)/2\}$,
$c_i(A) \ge n - k + 1 - t \ge i + \alpha n$.
Thus \eqref{c1} holds under every ordering.

Suppose that some ordering satisfies the weaker version of \eqref{p}. 

Since 
$d(S_j) = 1 + \ceil{\alpha n} \ge \min\{\min S_j,(n-k+2)/2\}+ 1$ for each $j \in [t]$, 
we have $\min S_j \le \ceil{\alpha n}$.
Since $S_1, \dots, S_t$ are pairwise disjoint, $\min S_1, \dots, \min S_t$ are distinct. But then $\min S_j \le \ceil{\alpha n}$ for every $j\in [t]$ implies $t \le \ceil{\alpha n}$, a contradiction with $t = 2\ceil{\alpha n}$.
\end{proof}

\subsection{Sharpness of
\texorpdfstring{\cref{thm:main}}{the Chvatal-type theorem}} \label{sec:sharp}

Throughout \cref{sec:sharp}, we fix $k\ge3$ and $0 < \alpha < 1/10$.

Recall that, in \cref{thm:main}, for each $A \in \binom{[n]}{k-2}$ and $1\le i<\min\{\min A,(n-k+2)/2\}$, the condition \eqref{c1} requires
\[c_i(A)\ge i+\alpha n,\]
and \eqref{c2} requires
\begin{center}
at least $i + \lceil\alpha n\rceil$ vertices $x$ satisfy 
$d(A\cup\{x\}) \ge n-k+2-i$.    
\end{center}
Using non-Hamiltonian examples, we shall see that none of the following weakenings of \eqref{c1} or \eqref{c2} could be possible:
\begin{lefttagged}
\begin{gather*}
 c_i(A)\ge i+ k-2;
 \tag{$C_1'$}\label{c1'}\\
 \text{at least }i+ k-2\text{ vertices }x\text{ satisfy }
 d(A\cup\{x\})\ge n-k+2-i;
 \tag{$C_2'$}\label{c2'}\\
 \text{at least }i+ \ceil{\alpha n}\text{ vertices }x\text{ satisfy }
 d(A\cup\{x\})\ge n-k+1-i.
 \tag{$C_2''$}\label{c2''}
\end{gather*}
\end{lefttagged}

To construct the desired examples, for a finite set $W$,
we require a family $\mathcal F \subseteq\binom{W}{2k-3}$ such that every $(k-2)$-subset of $W$ is contained in exactly one member of $\mathcal F$.
We call such a $\mathcal{F}$ a \emph{Steiner system}
$S(k-2,2k-3,|W|)$.
%
%For $k=2$, simply let $\mathcal{F}$ contain a set of a single vertex.
For $k=3$ and $\abs{W}$ divisible by $3$, take $\mathcal{F}$ to be a partition of $W$ into triples.
For every $k\ge4$, 
\cite[Theorem~1.4]{Keevash} guarantees the existence of such Steiner systems for infinitely many $|W|$.

We borrow the following constructions from Tuza \cite{Tuza}.

\begin{example}[\cite{Tuza}]\label[example]{ex:sharp}
Let $W$ be a finite set and $\mathcal{F}$ be an $S(k-2,2k-3,|W|)$ system.
Define $L_{\mathcal F}$ to be the $(k-1)$-graph on $W$ whose edges are all the $(k-1)$-subsets contained in an element of $\mathcal F$. 
Let $n=|W|+1$, and identify $W =[n]\setminus\set{1}$.
We define a $k$-graph
$G$ on vertex set $[n]$ by setting
\[
 E(G)
 = \set{\set{1}\cup e:e\in E(L_{\mathcal F})} \cup \binom {W}{k}.
\]
\end{example}

For sufficiently large $n$, such graphs $G$ are non-Hamiltonian, as proved in \cite[Lemma~4]{Tuza}.
For completeness, we include a short argument here.
%It is obvious when $k = 2$. 
Suppose $k \ge 3$ and $G$ had a
Hamiltonian cycle. 
There are $k$ ($k$-uniform) edges containing the vertex $1$ in the Hamiltonian cycle.
The $k$ $(k-1)$-uniform edges obtained by deleting $1$ from each of them form a tight path in $L_{\mathcal F}$ on $2k-2$ distinct vertices. 
Consecutive edges of this path share a $(k-2)$-set. 
The uniqueness property of
$\mathcal F$ therefore forces all the $(k-1)$-uniform edges to be contained in a single member of $\mathcal F$, which has only $2k-3$ vertices, a contradiction.

It remains to show that $G$ satisfies each proposed weakening.
For convenience, denote $m = n-k+2$.
For $A \in \binom{[n]}{k-2}$, if $1 \in A$, any condition upon $1\le i<\min\{\min A,m/2\} = 1$ is vacuously true. We hence assume $1 \notin A$ and consider the degree sequence of $L(A)$. 
As $A \subseteq W$, there is a unique $F \in \mathcal{F}$ such that $A \subseteq F$. Recall $\abs{F} = 2k-3$.
We have $d(A\cup\set{1}) = k-1$ as $A\cup\set{1}\cup\set{x}$ is an edge if and only if $x \in F \setminus A$.
For $v \in W\setminus F$, we have $d(A\cup\set{v}) = m-2$ as $A\cup\set{v}\cup\set{x}$ (with $x \notin A\cup\set{v}$) is an edge if and only if $x \neq 1$.
For $v \in F\setminus A$, we have $d(A\cup\set{v}) = m-1$.
Hence,
\[
 (c_1(A),\ldots,c_m(A))
 =\bigl(k-1,
 \underbrace{m-2,\ldots,m-2}_{m-k\text{ entries}},
 \underbrace{m-1,\ldots,m-1}_{k-1\text{ entries}}\bigr).
\]

\eqref{c1} cannot be weakened to \eqref{c1'} since the degree sequence of $L(A)$ always satisfies \eqref{c1'} or \eqref{c2}. Indeed, for every $1 \le i < \min\set{\min A, m/2}$, the inequality \eqref{c1'}, namely $c_i(A)\ge i+ k-2$, holds.

Similarly, \eqref{c2} cannot be weakened to \eqref{c2'} since the degree sequence of $L(A)$ always satisfies \eqref{c1} or \eqref{c2'}. Within this range, \eqref{c1} can fail only when $i = 1$. In that case, there are $k-1$ vertices $x$ such that $d(A\cup\{x\})\ge m-1$, so \eqref{c2'} holds.

Finally, \eqref{c2} cannot be weakened to \eqref{c2''}. 
Again, within this range, \eqref{c1} can fail only when $i = 1$. In that case, at least $1 + \ceil{\alpha n }$ vertices $x$ satisfy $d(A\cup\{x\})\ge m-2$, so \eqref{c2''} holds.

Although our Hamilton-framework method does not allow us to replace $\alpha n$ in \cref{thm:main} entirely by a constant (mainly because of the inheritance requirement), it is reasonable to conjecture the following exact form of \cref{thm:main}. 

\begin{conjecture}\label[conjecture]{conj:exact}
For every $k\ge3$, there exists $n_0$ such that the following holds for every $n\ge n_0$. Let $G$ be a $k$-graph on $[n]$. 
Suppose that, for every $A\in\binom{[n]}{k-2}$ and every integer $1\le i<\min\{\min A,(n-k+2)/2\}$, at least one of the following holds: 
\[
c_i(A)\ge i+k-1
\quad\text{or}\quad
c_{n-2k+4-i}(A)\ge n-k+2-i.
\]
Then, $G$ contains a (tight) Hamiltonian cycle.
\end{conjecture}
\noindent
Equivalently, the second inequality in \cref{conj:exact} requires
at least $i+k-1$ members of $\mathcal R(A)$ to be at least
$n-k+2-i$.

\subsection{Hamiltonian functions}\label{sec:hard}

\begin{definition}[Hamiltonian function]\label[definition]{def:hfun}
A function $D:\binom{[n]}{k-1} \rightarrow \set{0,\dots,n-k+1}$ is 
\emph{Hamiltonian} if every $k$-graph $G$ on $[n]$ satisfying $d_G(S)\ge D(S)$ for every $S\in\binom{[n]}{k-1}$
contains a Hamiltonian cycle.

For a rational $\alpha \ge 0$, we say $D$ is \emph{$\alpha$-Hamiltonian} if every $k$-graph $G$ on $[n]$ satisfying $d_G(S)\ge D(S) + \alpha n$ for every $S\in\binom{[n]}{k-1}$
contains a Hamiltonian cycle.
\end{definition}

We note that $D$ is Hamiltonian if and only if it is $0$-Hamiltonian.
Chv\'atal's classical theorem provides a complete characterisation of all Hamiltonian sequences (the $k=2$ case) \cite{Chvatal}. 
Sch\"ulke asked in \cite{Schuelke} for a complete characterisation of all the Hamiltonian matrices (the $k=3$ case). 
However,
we give a pessimistic answer to this question by showing that, for $k \ge 3$, determining whether a function is Hamiltonian is NP-hard.
In fact, we show a stronger result that even when a pointwise relaxation of $\alpha n$ is allowed, this question remains NP-hard.
Precisely, determining whether a function is $\alpha$-Hamiltonian is NP-hard for every $k \ge 3$ and all sufficiently small $\alpha$.
Therefore, we cannot hope for any equivalent characterisation that is checkable in P unless P $=$ NP. 

To put things formally, we define $\alpha$-HF$_k$ to be the collection of all the $\alpha$-Hamiltonian functions and write HF$_k=0$-HF$_k$.
We remark that for $\alpha \le \beta$, $\alpha$-HF$_k \subseteq \beta$-HF$_k$, and determining  whether $D \in \alpha\text{-}\mathrm{HF}_k$ is harder than determining whether
$D \in \beta\text{-}\mathrm{HF}_k$
(in the sense that we may determine whether $D \in \beta\text{-}\mathrm{HF}_k$ by asking whether $D + \ceil{\beta n} - \ceil{\alpha n} \in \alpha\text{-}\mathrm{HF}_k$).
In a slight abuse of notation, we may also use $\alpha$-HF$_k$ and HF$_k$ to denote the corresponding decision problems.

A \emph{promise problem} is a problem in which the input is promised to belong to exactly one of two cases.
The gap Hamiltonicity problem is a promise problem that asks whether an input graph is Hamiltonian or has a small ``gap'' of being Hamiltonian; see \cref{thm:gap-hamiltonicity}.
It is NP-hard as stated in the proof of \cite[Theorem~2]{BenderChekuri}.
We shall show $\alpha$-HF$_k$ is NP-hard by a reduction from the gap Hamiltonicity problem. 
For a graph $G$ on $n$ vertices and a cyclic vertex ordering
$\pi = (v_1,\dots,v_n)$, put $v_{n+1}=v_1$, and define
\[
g(\pi)
:=\abs{\set{i\in[n]:v_i v_{i+1}\notin E(G)}}.
\]

\begin{proposition}[Gap Hamiltonicity]
  \label[proposition]{thm:gap-hamiltonicity}
There exists a constant $\varepsilon>0$ for which the following promise problem is NP-hard.  
The input is a graph $G$ on $n$ vertices,
and the task is to distinguish between
\begin{enumerate}[label=\textup{(\roman*)}]
 \item $G$ contains a Hamiltonian cycle;
 \item $g(\pi)\ge\varepsilon n$ for every cyclic ordering $\pi$ of $V(G)$.
\end{enumerate}
\end{proposition}

We shall use the version of \cref{thm:gap-hamiltonicity} in which $n$ is even. It is a simple observation that the even-order version remains NP-hard (indeed, we may reduce any odd-order input to an even-order one by adding a new universal vertex).

We now show that recognising Hamiltonian functions is hard, already for functions taking only two values.

\begin{theorem}\label{thm:alpha-hard}
For every $k\ge3$, there exists $\alpha_k >0$ such that for every
rational $\alpha\in[0,\alpha_k]$, the decision problem
$\alpha$-HF$_k$ is NP-hard.
\end{theorem}

\begin{proof}
Let $0<\varepsilon<1$ be supplied by the even-order version of \cref{thm:gap-hamiltonicity} and set
$\alpha_k=\frac{\varepsilon}{10k^2}$.
Let $\alpha\in[0,\alpha_k]$ be rational, and let
$n$ be even and sufficiently large.
We reduce from the gap Hamiltonicity problem.

Let $G$ be a graph on vertex set $X$ with $\abs{X} = n$.
Let $Y$ be disjoint from $X$ with $\abs{Y} = (k-2)n/2$. 
Set $V = X \cup Y$ and $N = |V|$.
Define $D:\binom{V}{k-1}\rightarrow\mathbb{N}_{\ge0}$ by
\[
D(S) = 
\begin{cases}
N-k+1 - \ceil{\alpha N},&\text{if $S$ contains an edge of $G$},\\
0,&\text{otherwise}. 
\end{cases}
\]

Suppose first that $G$ has a Hamiltonian cycle
$x_1,x_2, \dots, x_n$. We want to show $D$ is $\alpha$-Hamiltonian.

Let $H_1$ be a $k$-graph such that $d(S) \ge D(S) + \ceil{\alpha N}$ for every $S \in \binom{V}{k-1}$. In particular, for $S$ that contains an edge of $G$, $d(S) \ge N - k + 1$.
Label vertices in $Y$ by
\[
y_{1,1},\dots,y_{1,k-2},\dots,y_{i,1},\dots,y_{i,k-2}
,\dots,y_{n/2,1},\dots,y_{n/2,k-2}.
\]
For every $i \in [n/2]$, insert $y_{i,1},\dots,y_{i,k-2}$ between $x_{2i-1}$ and $x_{2i}$ in $x_1, \dots, x_n$. The resulting cyclic ordering is a Hamiltonian cycle in $H_1$.

Suppose now that $g(\pi)\ge\varepsilon n$ for every cyclic ordering $\pi$ of $V(G)$. We want to show $D$ is not $\alpha$-Hamiltonian.

Fix $R\in\binom{Y}{\ceil{\alpha N}}$ and define the $k$-graph $H_0$ on $V$ by
\[
E(H_0)=\set{e\in\binom{V}{k}:
e\cap R\ne\emptyset
\text{ or $e$ contains an edge of $G$}}.
\]
Then, for any $S \in \binom{V}{k-1}$, $d(S) \ge \ceil{\alpha N}$. Moreover, for $S$ that contains an edge of $G$, $d(S) \ge N -k + 1$. 
Therefore, $d(S) \ge D(S) + \ceil{\alpha N}$ for every $S \in \binom{V}{k-1}$.
It suffices to show $H_0$ is not Hamiltonian. 
Suppose, for a contradiction, that $v_1,\dots,v_N$ is a Hamiltonian cycle in $H_0$.
 
For $i\in[N]$, let 
$W_i=\set{v_i,v_{i+1},\dots,v_{i+k-1}}$ and 
$q_i=\abs{W_i\cap X}$, where subscripts are read modulo $N$.
Let $I=\set{i\in[N]: W_i\cap R\ne\emptyset}$. 
Since every vertex of $R$ occurs in exactly $k$ of the sets $W_i$, 
$\abs{I}\le k\ceil{\alpha N}$.
For $i\notin I$, $W_i$ avoids $R$ and hence contains an edge of $G$, so $q_i\ge2$.  
Let
$J= \set{i\in[N]\setminus I:q_i\ge 3}$.
Since every vertex of $X$ occurs in exactly $k$ of $W_i$,
we have
$kn=\sum_{i=1}^N q_i \ge2(N-|I|)+|J| = kn-2|I| +|J|$.
Consequently,
$|J|\le2|I|$ and 
$\abs{I\cup J}\le 3k\ceil{\alpha N}$.

Delete the vertices of $Y$ from $(v_1,\ldots,v_N)$, obtaining a cyclic ordering $\pi=(x_1,\ldots,x_n)$ of $X$.  
For $j \in[n]$, 
let $i_j$ be defined by
$v_{i_j}=x_j$.  
If $i_j\notin I\cup J$, then $W_{i_j}$ avoids $R$ and contains exactly $x_j$ and $x_{j+1}$ of $X$. 
Since $W_{i_j}\in E(H_0)$, $x_j x_{j+1} \in E(G)$ for every $i_j\notin I\cup J$.
In other words, if $x_j x_{j+1}$ is not an edge in $G$, then $i_j \in I\cup J$.
Therefore,
$g(\pi) \le 3k\ceil{\alpha N}<\varepsilon n$,
a contradiction with $g(\pi) \ge \varepsilon n$.
Therefore, $D$ is not $\alpha$-Hamiltonian.

The reduction can be performed in polynomial time.
\end{proof}

\begin{proof}[Proof of \cref{thm:hard}]
This is a special case of \cref{thm:alpha-hard} in which $\alpha = 0$.
\end{proof}

\section*{Acknowledgements}
The author would like to thank Bjarne Sch\"ulke, Sim\'on Piga, and Nicol\'as Sanhueza-Matamala for very helpful discussions, and 
Peter Allen and Hong Liu for advice on revising the manuscript.
The author also thanks the IBS ECOPRO Group and the 2026 IBS ECOPRO Summer Research Program, both led by Hong Liu, for their hospitality.

\end{document}